\documentclass{article}
\usepackage{romp}
\usepackage{graphicx}
\usepackage{amsmath}
\usepackage{cite}
\usepackage[utf8]{inputenc}
\usepackage[T1]{fontenc}
\usepackage{caption}

\usepackage[section]{placeins}
\DeclareMathOperator{\sgn}{sgn}

\title{Quantum logistic map considered as discrete-time Heisenberg equation}
\author{ Maciej Janowicz\thanks{Corresponding author.} \\ Department of Applied Mathematics, Institute of Information Technology, \\ Warsaw University of Life Sciences - SGGW,\\ ul. Ciszewskiego 8, 02-786 Warsaw, Poland  \\ e-mail: maciej\_janowicz@sggw.edu.pl  \\[2ex]
         Arkadiusz Orłowski
                      \\ Institute of Mathematics and Cryptology, \\ Faculty of Cybernetics, \\ Military University of Technology, \\
                      ul. gen. Sylwestra Kaliskiego 2, 00--908 Warsaw, Poland\\ e-mail: arkadiusz.orlowski@wat.edu.pl}

\begin{document}

\maketitle

\begin{abstract}
The scalar iterates of the logistic map are represented as multiplication
operators on $L^2([0,1])$, and their fixed matrix elements in the normalized
shifted-Legendre basis are investigated.  Three parameter regimes admit
complete analytical control.  At $r=5/2$, every fixed matrix element
converges to $3\delta_{kl}/5$, where $\delta_{kl}$ is the Kronecker delta.
At $r=16/5$, an exact phase-basin decomposition of the attracting period-two
orbit yields phase-resolved limits for the even and odd subsequences.  At
$r=4$, an exact Chebyshev-moment representation shows that every fixed matrix
element approaches $\delta_{kl}/2$ at the rate $O(4^{-n})$.  The chaotic case
$r=37/10$ is examined solely through a controlled finite numerical refinement
study.  The analytical statements concern fixed matrix elements for specified
basis indices and do not imply operator-norm convergence.  Complementary
finite-time diagnostics include matrix-element time series, a
bifurcation-style diagram, a time-averaged mean intensity, a normalized
second-order intensity moment, and normalized scalar OTOC-type commutator
correlation matrices.  A separate finite-dimensional recursion
$X_{k+1}=R X_k(I-X_k)R^\dagger$, with a fixed matrix $R$ and diagonal or
tridiagonal amplitude profiles, is explored by finite-time numerics.  A
regularized phase-space lift is included as a controlled visualization.
\end{abstract}


\section{Introduction}

\noindent
The logistic map serves as a paradigmatic model of low-dimensional
nonlinear dynamics~\cite{ColletEckmann,Cvitanovic,Devaney,EckmannRuelle1985,Feigenbaum1978,KantzSchreiber,May1976,Ott,Strogatz}.  Here, its scalar
iterates are used to define a family of multiplication operators.  This keeps
the nonlinear dynamics at the level of functions on $[0,1]$ while allowing
exact statements about fixed matrix elements in a standard orthonormal basis.
Matrix-valued logistic maps provide a related but distinct direction of
generalization~\cite{PZ2025}.  For prior use of the term ``quantum logistic
map,'' see Ref.~\cite{Goggin1990QuantumLogisticMap}.  More broadly, the
relation between classical and quantum chaos
has long been studied through semiclassical structure and spectral or
dynamical diagnostics~\cite{Berry1977,Gutzwiller,Haake}.

For geometric context, the local canonical phase-space lift of the scalar
map and a regularized visualization of a procedural cat silhouette are also
considered.  The regularization removes the critical-line singularity but is
explicitly neither symplectic nor area preserving.

The contributions concern four parameter regimes.  At $r=5/2$, almost-everywhere
convergence to the attracting fixed point gives a limit for every fixed basis
element.  At $r=16/5$, an exact phase-basin decomposition gives the two
subsequence limits associated with the attracting period-two orbit.  At
$r=4$, a finite Chebyshev representation yields an exact evaluation and the
fixed-element asymptotic $O_{kl}(4^{-n})$.  At $r=37/10$, the conclusions are
restricted to finite, controlled numerical evidence: no theorem about an
invariant or physical measure, mixing, ergodicity, or continuum convergence is
claimed.  Across these regimes, numerical diagnostics identify the loss of
resolution that can affect fixed global quadrature rules at late iterations.

These results concern the multiplication-operator dynamics and fixed basis
elements.  They do not establish operator-norm convergence.  A separate
operator-valued recursion involving a fixed operator $R$ is examined through
exploratory finite-dimensional, finite-time numerics; no general claim about
its boundedness or asymptotic behavior is made here.

\section{Mathematical setup}

\subsection{Classical iterates and multiplication operators}

For $0<r\leq4$, define
\begin{equation}
 p_0(u;r)=u,\qquad
 p_{n+1}(u;r)=r\,p_n(u;r)\bigl(1-p_n(u;r)\bigr),
 \qquad u\in[0,1].
 \label{eq:scalar-logistic-iterates}
\end{equation}
The parameter argument will be suppressed when no ambiguity can arise.

Let
\[
 \mathcal H=L^2([0,1],du)
\]
and let $M_u$ denote multiplication by the coordinate,
\[
 (M_u f)(u)=u f(u).
\]
The multiplication-operator iterates are defined by
functional calculus as
\begin{equation}
 X_n=p_n(M_u;r)=M_{p_n(\cdot;r)}.
 \label{eq:multiplication-operator-iterates}
\end{equation}
Thus $X_n$ is multiplication by the scalar function $p_n(\cdot;r)$;
this is not a finite-matrix recursion. In the limited terminology used
in the title, the nonlinear operator update
\[
X_{n+1}=rX_n(I-X_n)
\]
is regarded as a discrete-time Heisenberg-type equation for the
observable $X_n$. It is not generated by unitary conjugation and should
not be confused with the standard Heisenberg equation of motion.

\subsection{Canonical phase-space lift and regularized visualization}

Introducing an auxiliary momentum coordinate $\pi_n$, the scalar update has
the local phase-space lift
\begin{equation}
 \begin{aligned}
  x_{n+1}&=r x_n(1-x_n),\\
  \pi_{n+1}&=\frac{\pi_n}{r(1-2x_n)}.
 \end{aligned}
 \label{phase_space_eqns}
\end{equation}
For $x_n\neq1/2$, its Jacobian matrix is triangular up to the lower-left
entry, and its determinant is
\[
 r(1-2x_n)\frac{1}{r(1-2x_n)}=1.
\]
Equation~\eqref{phase_space_eqns} is therefore a local canonical,
area-preserving lift away from the critical line.  This statement does not
imply that the map is globally invertible.  The momentum update is singular
at $x_n=1/2$.

For the visualization in Figure~\ref{fig:fig1}, the regularized update
\begin{equation}
 \begin{aligned}
  x_{n+1}&=r x_n(1-x_n),\\
  \pi_{n+1}&=
  \frac{\pi_n\,\operatorname{sgn}(1-2x_n)}
  {r\sqrt{(1-2x_n)^2+\epsilon^2}},
 \end{aligned}
 \label{eq:regularized-phase-space}
\end{equation}
is used instead, with $\operatorname{sgn}(0)=0$.  This prescription keeps the update finite
on the critical line but does not make it continuous there.  Its Jacobian determinant is
\[
 \frac{|1-2x_n|}{\sqrt{(1-2x_n)^2+\epsilon^2}}.
\]
Consequently, for $\epsilon>0$ the regularized map is not exactly symplectic
or area preserving.

\begin{figure}[htbp]
 \centering
 \includegraphics[width=\textwidth]{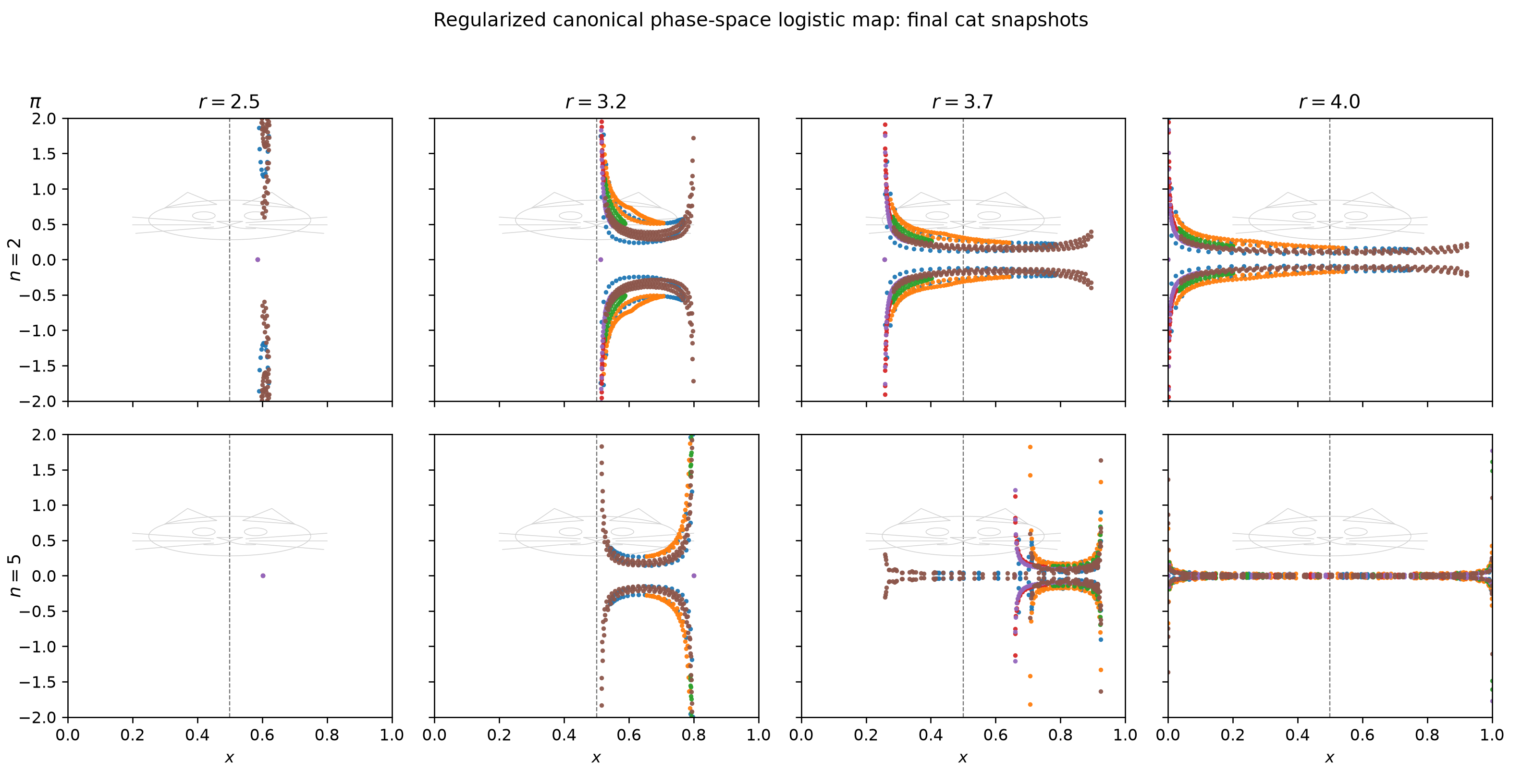}
 \caption{Regularized phase-space logistic-map snapshots for
 $r=5/2,16/5,37/10$, and $4$ (columns), after $n=2$ (top row) and $n=5$
 (bottom row).  The plotted points originate from a deterministic procedural
 cat silhouette; its faint initial outline is included in every panel for
 reference.  The regularized map in Eq.~\eqref{eq:regularized-phase-space}
 is used with $\epsilon=10^{-6}$.  The display window is fixed at
 $x\in[0,1]$ and $\pi\in[-2,2]$, so evolved points outside the momentum
 window are clipped; the dashed vertical line marks $x=1/2$.}
 \label{fig:fig1}
\end{figure}

\subsection{Shifted-Legendre representation}

Let $P_j$ be the standard Legendre polynomial of degree $j$.  The normalized
shifted-Legendre basis
\begin{equation}
 \phi_j(u)=\sqrt{2j+1}\,P_j(2u-1),\qquad j=0,1,2,\ldots,
 \label{eq:shifted-legendre-basis}
\end{equation}
is used; it is orthonormal in $\mathcal H$.  In this basis, $M_u$ has the symmetric
tridiagonal matrix
\begin{equation}
 \langle \phi_m,M_u\phi_n\rangle=
 \begin{cases}
 \displaystyle\frac12, & m=n,\\[2mm]
 \displaystyle\frac{n+1}{2\sqrt{(2n+1)(2n+3)}}, & m=n+1,\\[3mm]
 \displaystyle\frac{n}{2\sqrt{(2n-1)(2n+1)}}, & m=n-1,\\[3mm]
 0, & \text{otherwise}.
 \end{cases}
 \label{matrix_elements_zero}
\end{equation}
The symmetry is also explicit in the equivalent off-diagonal formula
$\langle\phi_n,M_u\phi_{n+1}\rangle=(n+1)/
[2\sqrt{(2n+1)(2n+3)}]$.

For fixed nonnegative integers $k,l$, define
\begin{equation}
 x_{kl}(n;r)=\langle\phi_k,X_n\phi_l\rangle
 =\int_0^1\phi_k(u)p_n(u;r)\phi_l(u)\,du.
 \label{eq:fixed-basis-matrix-elements}
\end{equation}
These are matrix elements of a multiplication operator, not elements of a
density matrix.  If $P_N$ denotes the orthogonal projection onto
$\operatorname{span}\{\phi_0,\ldots,\phi_{N-1}\}$, then the corresponding
finite matrix is the representation of $P_NX_nP_N$.  Applying the nonlinear
logistic recursion directly to the compression $P_NM_uP_N$ is not
automatically equivalent to compressing the functional-calculus result.  In
general, with the relevant identity operators and domains understood,
\begin{equation}
 p_n(P_NM_uP_N;r)\neq P_Np_n(M_u;r)P_N.
 \label{eq:projection-caveat}
\end{equation}

\subsection{A separate operator-valued extension}

For motivation, one may separately consider
\begin{equation}
 \mathcal X_{n+1}=R\bigl(\mathcal X_n-\mathcal X_n^2\bigr)R^\dagger,
 \label{eq:operator-extension-setup}
\end{equation}
where $R$ is fixed.  This is not the definition of the multiplication-operator
sequence in Eq.~\eqref{eq:multiplication-operator-iterates}.  In the scalar
control case $R=sI$, Eq.~\eqref{eq:operator-extension-setup} becomes
\[
 \mathcal X_{n+1}=|s|^2\bigl(\mathcal X_n-\mathcal X_n^2\bigr),
\]
so the effective scalar parameter is $r=|s|^2$; a positive scalar parameter
$r$ corresponds to amplitude $\sqrt r$, up to phase.  The detailed proved
 results below concern Eq.~\eqref{eq:multiplication-operator-iterates} and the
 fixed elements in Eq.~\eqref{eq:fixed-basis-matrix-elements}.  A separate
 finite-dimensional reconnaissance of the extension is given in
 Section~\ref{sec:finite-operator-extension}.

\section{Analytically controlled regular regimes}

\subsection{Attracting fixed point at $r=5/2$}

\begin{proposition}{Proposition}
\label{prop:r25-fixed-element}
For the recursion in Eq.~\eqref{eq:scalar-logistic-iterates} with $r=5/2$,
\[
 p_n(u;5/2)\longrightarrow\frac35
\]
for Lebesgue-almost every $u\in[0,1]$.  Consequently, for every fixed pair
of nonnegative indices $k,l$,
\begin{equation}
 \lim_{n\to\infty}x_{kl}(n;5/2)=\frac35\delta_{kl}.
 \label{eq:r25-fixed-element-limit}
\end{equation}
\end{proposition}

\noindent\textit{Proof.}
Let $f(u)=(5/2)u(1-u)$.  Its nonzero fixed point is $u_*=3/5$, and the
standard fixed-point dynamics gives $f^n(u)\to u_*$ for every
$u\in(0,1)$.  The endpoints are exceptional: $0$ is fixed and $1$ maps to
$0$.  Thus the endpoint orbit and all its preimages in $[0,1]$ form here the
finite set $\{0,1\}$, which has Lebesgue measure zero.  Since $[0,1]$ is
forward invariant, $0\leq p_n(u;5/2)\leq1$.  For fixed $k,l$, the continuous
function $|\phi_k\phi_l|$ is integrable and dominates
$|\phi_k p_n\phi_l|$.  Dominated convergence and orthonormality therefore
give
\[
 \lim_{n\to\infty}x_{kl}(n;5/2)
 =\frac35\int_0^1\phi_k(u)\phi_l(u)\,du
 =\frac35\delta_{kl}.
\]
No quantitative rate follows from this argument; in particular, the local
multiplier at $3/5$ does not control long endpoint transients uniformly.
The proposition also makes no assertion of operator-norm convergence.
This completes the proof.

For finite iterations, two endpoint-resolving procedures are compared,
namely adaptive quadrature with an explicit symmetric endpoint partition and a
composite Gauss rule on geometrically shrinking endpoint intervals.  Their
agreement is an independent numerical check, not a proof of
Eq.~\eqref{eq:r25-fixed-element-limit}.  A fixed global Gauss rule becomes
unreliable once the late-iteration integrand develops structure finer than
its fixed node spacing.

\subsection{Exact period-two phases at $r=16/5$}

Set
\begin{equation}
 f(u)=\frac{16}{5}u(1-u),\qquad
 q=\frac{11}{16},\qquad
 a=\frac{21-\sqrt{21}}{32},\qquad
 b=\frac{21+\sqrt{21}}{32}.
 \label{eq:r32-cycle-points}
\end{equation}
Then $q$ is the nonzero fixed point, $f(a)=b$, $f(b)=a$, and the multiplier
of the attracting two-cycle is
\[
 f'(a)f'(b)=\frac4{25}.
\]
Define
\begin{equation}
 c_1=\frac5{16},\qquad
 c_{j+1}=\frac{1-\sqrt{1-(5/4)c_j}}2,
 \qquad j\geq1,
 \label{eq:r32-boundary-sequence}
\end{equation}
and introduce the open intervals
\begin{equation}
 I_j=(c_{j+1},c_j),\qquad
 I_j^*=(1-c_j,1-c_{j+1}).
 \label{eq:r32-phase-intervals}
\end{equation}

\begin{theorem}{Theorem}
\label{thm:r32-phase-basins}
Let
\begin{equation}
 E=\{0,1,q\}\cup\{c_j,1-c_j:j\geq1\},
 \label{eq:r32-exceptional-set}
\end{equation}
with duplicates removed.  For every $u\in[0,1]\setminus E$, the even and
odd limits of $p_n(u;16/5)$ exist.  On the central interval $(c_1,q)$,
\[
 p_{2m}(u;16/5)\to a,\qquad p_{2m+1}(u;16/5)\to b.
\]
On $I_j$ and $I_j^*$ the limiting phase is $(b,a)$ for odd $j$ and $(a,b)$
for even $j$, where the ordered pair denotes the even and odd limits.
\end{theorem}

\noindent\textit{Proof.}
The lower inverse branch of $f$ is
\[
 h(y)=\frac{1-\sqrt{1-(5/4)y}}2.
\]
It satisfies $f(c_1)=q$, $f(c_{j+1})=c_j$, and $0<h(y)<y$ for
$0<y<q$.  Hence $c_j\downarrow0$.  The two inverse images of $q$ are
$c_1$ and $q=1-c_1$, and the inverse images of each $c_j$ are
$c_{j+1}$ and $1-c_{j+1}$.  For $j\geq2$, $1-c_j>4/5$, so it has no
real inverse image under $f$, whose maximum on $[0,1]$ is $4/5$.
Consequently, Eq.~\eqref{eq:r32-exceptional-set} consists of the endpoints
and the complete backward orbit of $q$.  It is countable and therefore has
Lebesgue measure zero.

Let $g=f^2$.  Direct factorization gives
\begin{equation}
 g(u)-u=-\left(\frac{16}{5}\right)^3u(u-a)(u-q)(u-b).
 \label{eq:r32-second-iterate-factorization}
\end{equation}
On $[c_1,q]$, the only critical point of $g$ is $1/2$,
$g(c_1)=g(q)=q$, and $g(1/2)=f(4/5)=64/125$.  Thus
$g([c_1,q])=[64/125,q]$.  On this image, both $u$ and $f(u)$ exceed
$1/2$, so $g'(u)=f'(f(u))f'(u)>0$.  The factorization shows that
$g(u)>u$ on $(1/2,a)$ and $g(u)<u$ on $(a,q)$.  Since the increasing map
$g$ fixes $a$, every orbit in $[64/125,q)$ converges monotonically to $a$.
It follows that the central interval has phase $(a,b)$.

On the increasing left branch, $f$ maps $I_1$ onto $(c_1,q)$ and maps
$I_j$ onto $I_{j-1}$ for $j\geq2$.  Each passage to the preceding interval
reverses the alignment of the original even and odd subsequences, producing
the stated alternation.  Finally, $f(u)=f(1-u)$, so each reflected interval
$I_j^*$ has the same phase as $I_j$.  These intervals and the central
interval are precisely the connected components of $[0,1]\setminus E$.
This completes the proof.

Define the phase-basin set
\begin{equation}
 S_{ba}=\bigcup_{\substack{j\geq1\\ j\ \text{odd}}}(I_j\cup I_j^*).
 \label{eq:r32-sba}
\end{equation}
Outside $E$, Theorem~\ref{thm:r32-phase-basins} gives the pointwise limits
\[
 A(u)=a+(b-a)\mathbf 1_{S_{ba}}(u),\qquad
 B(u)=b-(b-a)\mathbf 1_{S_{ba}}(u).
\]
For fixed $k,l$, define the subsequential matrix-element limits
\[
 x_{kl}^{\rm even}(16/5)=\lim_{m\to\infty}x_{kl}(2m;16/5),\qquad
 x_{kl}^{\rm odd}(16/5)=\lim_{m\to\infty}x_{kl}(2m+1;16/5).
\]
Because $0\leq p_n\leq1$ and $\phi_k\phi_l$ is integrable, dominated
convergence yields
\begin{align}
 x_{kl}^{\rm even}(16/5)
 &=a\delta_{kl}+(b-a)\int_{S_{ba}}\phi_k(u)\phi_l(u)\,du,\label{eq:r32-even-limit}\\
 x_{kl}^{\rm odd}(16/5)
 &=b\delta_{kl}-(b-a)\int_{S_{ba}}\phi_k(u)\phi_l(u)\,du.\label{eq:r32-odd-limit}
\end{align}
Since $a+b=21/16$,
\begin{equation}
 x_{kl}^{\rm even}(16/5)+x_{kl}^{\rm odd}(16/5)
 =\frac{21}{16}\delta_{kl}.
 \label{eq:r32-limit-sum}
\end{equation}
Values assigned on the countable boundary set $E$ do not affect these
Lebesgue integrals.  Equations~\eqref{eq:r32-even-limit}--\eqref{eq:r32-limit-sum}
are fixed-element subsequential limits for fixed $k,l$, not operator-norm
limits, and no quantitative convergence rate is asserted.

The limiting basin integral can be evaluated interval by interval using a
polynomial antiderivative of $\phi_k\phi_l$.  If intervals through index $J$
are retained, the unrepresented endpoint tails obey
\[
 |\Delta I_{kl}|\leq
 2c_{J+1}\sqrt{(2k+1)(2l+1)},
\]
because $|P_j(y)|\leq1$ on $[-1,1]$.  Reflected intervals are evaluated via
\[
 \phi_k(1-u)\phi_l(1-u)=(-1)^{k+l}\phi_k(u)\phi_l(u),
\]
which avoids subtracting nearly equal antiderivative values close to $u=1$.
Independent adaptive and partitioned composite calculations agree for the
audited finite iterations.  Fixed global Gauss quadrature is therefore not
used as a reference method in this regime.

\section{Chaotic finite refinement and an exactly solvable endpoint}

\subsection{Finite numerical refinement at $r=37/10$}

For fixed nonnegative indices $k,l$, the raw ensemble matrix element is
\begin{equation}
 x_{kl}(n;37/10)
 =
 \int_0^1
 \phi_k(u)\,p_n(u;37/10)\,\phi_l(u)\,du .
 \label{eq:r37-raw-element}
\end{equation}
The finite quantities considered here are its deterministic quadrature
approximations, their Ces\`aro averages through a horizon $T$,
\begin{equation}
 \overline{x}_{kl}^{(N,s)}(T)
 =
 \frac{1}{T+1}\sum_{n=0}^{T}x_{kl}^{(N,s)}(n),
 \qquad
 x_{kl}^{(N,s)}(n)
 =
 \frac{1}{N}\sum_{j=0}^{N-1}
 \phi_k(u_j)\,p_n(u_j;37/10)\,\phi_l(u_j),
 \label{eq:r37-finite-cesaro}
\end{equation}
and the inclusive late-window averages
\begin{equation}
 L_{kl}^{(N,s)}(a,b)
 =
 \frac{1}{b-a+1}\sum_{n=a}^{b}x_{kl}^{(N,s)}(n),
 \qquad
 u_j=\frac{j+s}{N}.
 \label{eq:r37-late-window}
\end{equation}
Transfer-operator, invariant-measure, resonance, and finite-rank
approximation methods provide standard background for such questions
\cite{Baladi,LasotaYorke,Ruelle1986,Ulam}; the calculations below remain
finite numerical diagnostics and do not invoke those results as proofs.
Deterministic equal-stratum sampling was used with
$N=2^{15},2^{16},2^{17},2^{18}$ and within-stratum offsets
$s=1/4,1/2,3/4$; the choice $s=1/2$ is the midpoint rule.  The horizons
were $T=200,400,800$, with corresponding late windows $120$--$200$,
$240$--$400$, and $480$--$800$.  Conservative transfer-grid comparisons
used $1024,2048,4096$, and $8192$ cells.

For the fixed pairs $(0,0)$, $(5,5)$, and $(4,10)$, the raw late values
are not resolved under refinement.  By contrast, the finite Ces\`aro and
late-window means are substantially more stable.  The envelopes obtained
by changing the three deterministic offsets are refinement diagnostics,
not confidence intervals.  The transfer discretization conserves the
signed mass to the observed numerical tolerance (the largest recorded
residual was $5.66\times10^{-15}$), but it also introduces artificial
diffusion and damping.  Moreover, horizon drift remains material for some
Ces\`aro averages, even where the late-window means change much less.

These are finite, conditional observations.  They do not establish
convergence of the raw ensemble sequence, a finite period, convergence of
the Ces\`aro averages as $T\to\infty$, a zero continuum off-diagonal limit,
or existence or uniqueness of an invariant or physical measure.  Nor do
they establish mixing or ergodicity at exactly $r=37/10$, equality among
ensemble, single-orbit time, Ces\`aro, and invariant-measure averages, or
operator-norm convergence.  In particular, finite stabilization under the
stated refinements is not a proof of an asymptotic limit.

\subsection{Exact fixed matrix elements at $r=4$}

At the endpoint, let
\begin{equation}
 p_0(u;4)=u,\qquad
 p_{n+1}(u;4)=4p_n(u;4)\bigl(1-p_n(u;4)\bigr),
\end{equation}
and, for fixed nonnegative $k,l$, define
\begin{equation}
 x_{kl}(n;4)
 =
 \int_0^1\phi_k(u)\,p_n(u;4)\,\phi_l(u)\,du .
 \label{eq:r4-matrix-element}
\end{equation}

\begin{theorem}{Theorem}
\label{thm:r4-fixed-elements}
Put $d=k+l$, $S_{kl}=\sqrt{(2k+1)(2l+1)}$, and let
\begin{equation}
 P_k(y)P_l(y)=\sum_{m=0}^{d}a_mT_m(y),
 \qquad
 A_{kl}=\sum_{m=0}^{d}|a_m|.
 \label{eq:r4-cheb-expansion}
\end{equation}
For every $n\geq0$, with $N=2^n$,
\begin{equation}
 x_{kl}(n;4)
 =
 \frac{(-1)^{k+l}S_{kl}}{4}
 \sum_{m=0}^{d}a_m
 \left[
 \mu_m-\frac{\mu_{m+N}+\mu_{|m-N|}}{2}
 \right],
 \label{eq:r4-exact-formula}
\end{equation}
where
\begin{equation}
 \mu_j=\int_{-1}^{1}T_j(y)\,dy
 =
 \begin{cases}
 0,&j\ \text{odd},\\[2mm]
 \displaystyle\frac{2}{1-j^2},&j\ \text{even}.
 \end{cases}
 \label{eq:r4-cheb-moments}
\end{equation}
Thus the case $j=1$ belongs to the odd branch and entails no singular
interpretation.  For $N>d$,
\begin{equation}
 \left|x_{kl}(n;4)-\frac12\delta_{kl}\right|
 \leq
 \frac{2S_{kl}A_{kl}}{3(N-d)^2},
 \label{eq:r4-explicit-bound}
\end{equation}
and consequently, for fixed $k,l$,
\begin{equation}
 x_{kl}(n;4)=\frac12\delta_{kl}+O_{kl}(4^{-n}).
 \label{eq:r4-fixed-asymptotic}
\end{equation}
\end{theorem}

\textit{Proof.}
Write $u=\sin^2\theta$, $0\leq\theta\leq\pi/2$.  The endpoint-inclusive
identity
\begin{equation}
 p_n(u;4)=\sin^2(2^n\theta)
 \label{eq:r4-conjugacy}
\end{equation}
holds for all $n\geq0$, including both endpoints.  Now set
$t=2\theta$, $y=\cos t$, and $N=2^n$.  Then
\[
 2u-1=-y,\qquad
 p_n(u;4)=\frac{1-T_N(y)}{2},\qquad
 du=\frac12\sin t\,dt .
\]
Using $P_j(-y)=(-1)^jP_j(y)$ and the normalization of the shifted
Legendre basis gives
\begin{equation}
 x_{kl}(n;4)
 =
 \frac{(-1)^{k+l}S_{kl}}{4}
 \int_{-1}^{1}P_k(y)P_l(y)\bigl(1-T_N(y)\bigr)\,dy .
 \label{eq:r4-transformed-integral}
\end{equation}
The finite expansion \eqref{eq:r4-cheb-expansion}, together with
\[
 T_m(y)T_N(y)
 =
 \frac{T_{m+N}(y)+T_{|m-N|}(y)}{2},
\]
therefore yields \eqref{eq:r4-exact-formula}.  Direct integration gives
\eqref{eq:r4-cheb-moments}: odd Chebyshev polynomials integrate to zero,
while the even formula includes $j=0$ and the apparently exceptional
$j=1$ is already covered by oddness.

The nonoscillatory part of \eqref{eq:r4-transformed-integral} is
\[
 \frac{(-1)^{k+l}S_{kl}}{4}
 \int_{-1}^{1}P_k(y)P_l(y)\,dy
 =\frac12\delta_{kl}
\]
by Legendre orthogonality.  If $q\geq2$, then
$|\mu_q|=2/(q^2-1)\leq8/(3q^2)$.  When $N>d$, both $m+N$ and
$|m-N|=N-m$ are at least $N-d$; any odd index contributes zero, and
the even nonzero indices are at least two.  Hence the oscillatory
remainder satisfies
\[
 \left|x_{kl}(n;4)-\frac12\delta_{kl}\right|
 \leq
 \frac{S_{kl}}{8}\sum_{m=0}^{d}|a_m|
 \left(|\mu_{m+N}|+|\mu_{N-m}|\right)
 \leq
 \frac{2S_{kl}A_{kl}}{3(N-d)^2}.
\]
If $d\geq1$ and $N\geq2d$, this is at most
$8S_{kl}A_{kl}/(3N^2)$, which is a constant times $4^{-n}$.
For $d=0$, the exact formula directly gives
$x_{00}(n;4)=\frac12-\mu_N/4$ and the same $N^{-2}$ conclusion.
This completes the proof.

The theorem concerns fixed matrix elements.  It does not imply
operator-norm convergence, uniformity when $k$ or $l$ grows with $n$,
correctness of a nonlinear iteration performed after finite projection,
or a density-matrix interpretation.  Algebraic Chebyshev moments provide
the reference numerical evaluation, while cosine-weighted adaptive
quadrature supplies an independent check.  Their agreement for the
examined fixed elements supports the implementation but is not part of the
proof.  In contrast, fixed global Gauss quadrature eventually
under-resolves the increasing oscillation and can produce apparent
collapse or revival that is a quadrature artifact.

\section{Finite numerical diagnostics across parameter regimes}

The fixed-parameter results above are supplemented with three explicitly
finite diagnostics.  Figure~\ref{fig:matrix-elements-time} first compares the common
selection $x_{00}(n;r)$, $x_{22}(n;r)$, $x_{55}(n;r)$, and
$x_{4,10}(n;r)$ over the four regimes treated above.

\begin{figure}[htbp]
 \centering
 \includegraphics[width=0.94\textwidth]{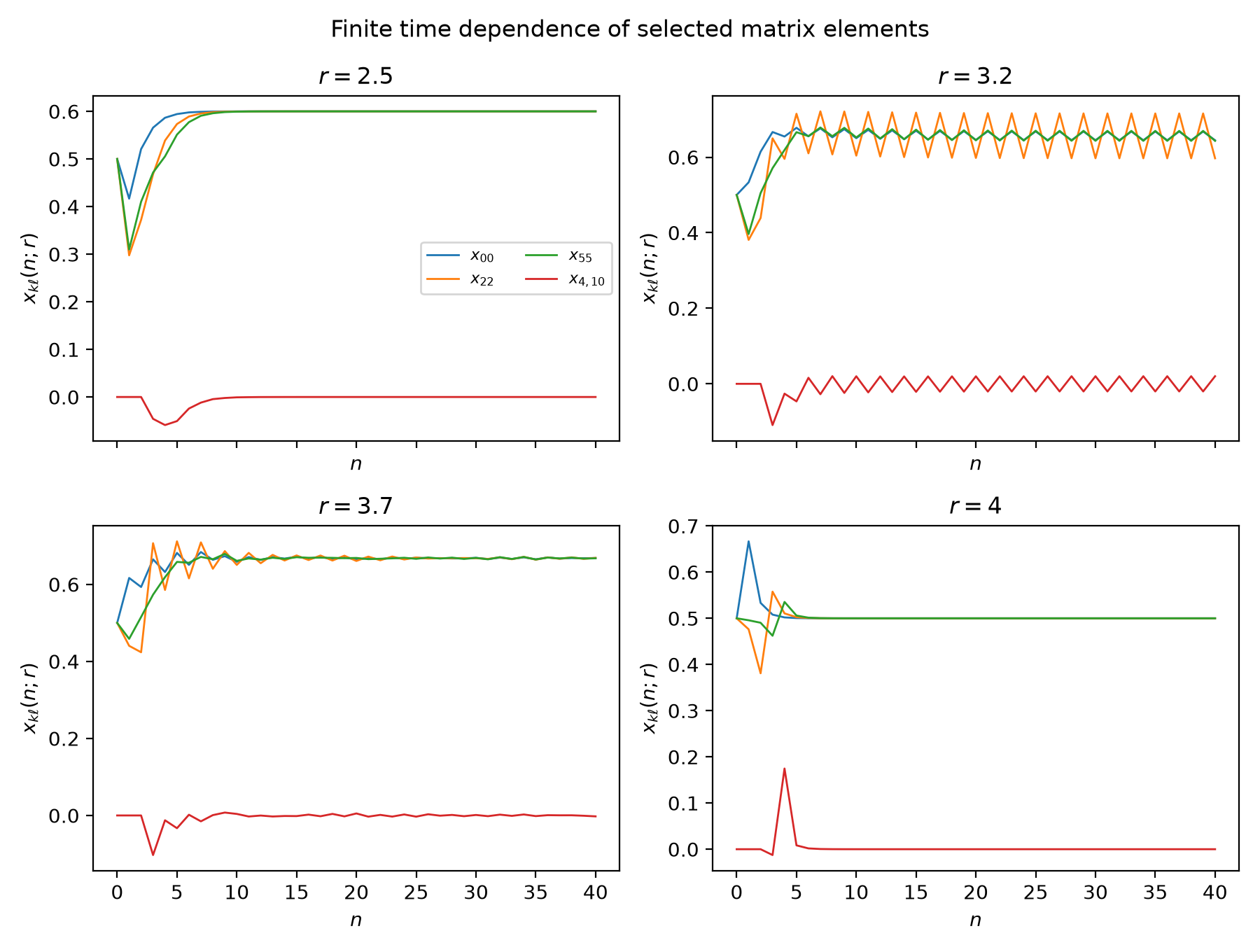}
 \caption{Finite time dependence of $x_{00}$, $x_{22}$, $x_{55}$, and
 $x_{4,10}$ for $r=5/2$, $16/5$, $37/10$, and $4$, displayed for
 $n=0,\ldots,40$.  The $r=5/2$ panel uses endpoint-partitioned composite
 Gauss quadrature, the $r=16/5$ panel uses the period-two boundary
 decomposition with composite Gauss quadrature, and the $r=4$ panel uses
 the exact Chebyshev-moment implementation.  These regular regimes are
 analytically controlled.  The $r=37/10$ panel is instead a deterministic
 finite equal-stratum diagnostic (principally $N=65536$, with a faint
 $N=131072$, offset-$1/4$ refinement envelope); no asymptotic limit of its
 raw sequence is claimed.}
 \label{fig:matrix-elements-time}
\end{figure}

The second diagnostic, presented in Figure~\ref{fig:matrix-bifurcation-x22},
is based on
\begin{equation}
 x_{22}(n;r)=\int_0^1\phi_2(u)p_n(u;r)\phi_2(u)\,du .
 \label{eq:x22-finite-diagnostic}
\end{equation}
The lower panel of Figure~\ref{fig:matrix-bifurcation-x22} reports only the
classical finite-orbit quantity $\lambda_T(r)$; it is not a distinct
``quantum Lyapunov exponent.''

\begin{figure}[htbp]
 \centering
 \includegraphics[width=0.82\textwidth]{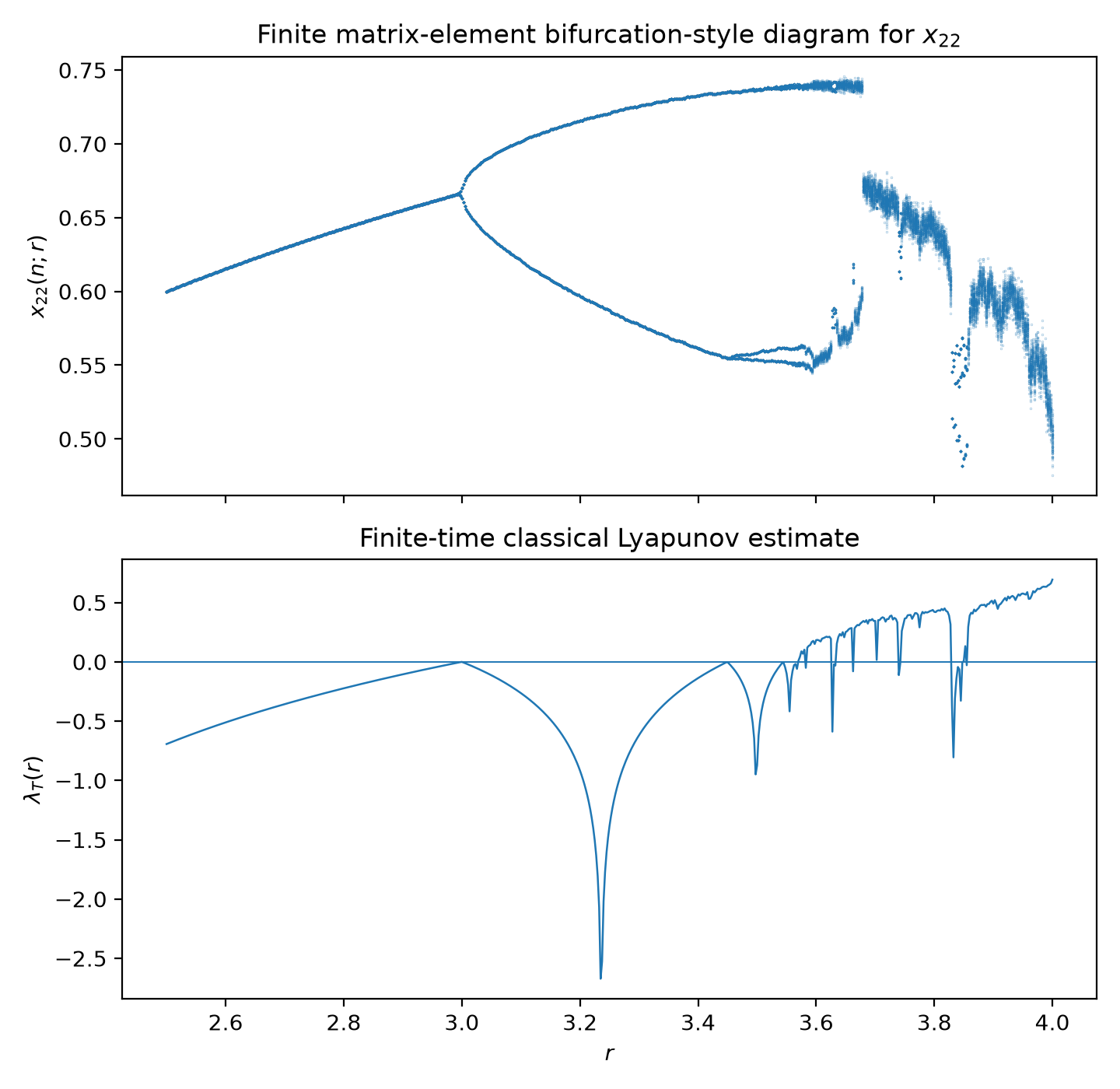}
 \caption{Upper panel: finite matrix-element bifurcation-style diagram for
 $x_{22}(n;r)$ on a grid of 601 equally spaced values of $r\in[2.5,4]$.
 Equal-stratum quadrature uses $N=8192$ and midpoint offset $s=1/2$; 500
 updates are discarded and updates 501--580 are retained.  Lower panel: one
 finite-time classical Lyapunov estimate $\lambda_T(r)$ from
 $x_0=0.123456789$, after 1000 discarded updates and over 2000 retained
 states.  The chaotic-region point cloud is a finite-resolution
 visualization, with no claim of continuum convergence.}
 \label{fig:matrix-bifurcation-x22}
\end{figure}

Finally, for a normalized state $\Psi$, define
\begin{equation}
 I_{\Psi}(n,r)=\int_0^1 |\Psi(u)|^2p_n(u;r)^2\,du,
 \qquad
 \overline{I}_{\Psi}(T,r)=\frac{1}{T+1}\sum_{n=0}^{T}I_{\Psi}(n,r).
 \label{eq:finite-mean-intensity}
\end{equation}
Figure~\ref{fig:mean-intensity-vs-r} uses $\Psi=\phi_0$, $T=400$, and
deterministic equal-stratum quadrature with $N=8192$ and midpoint offset
$s=1/2$.

\begin{figure}[htbp]
 \centering
 \includegraphics[width=0.96\textwidth]{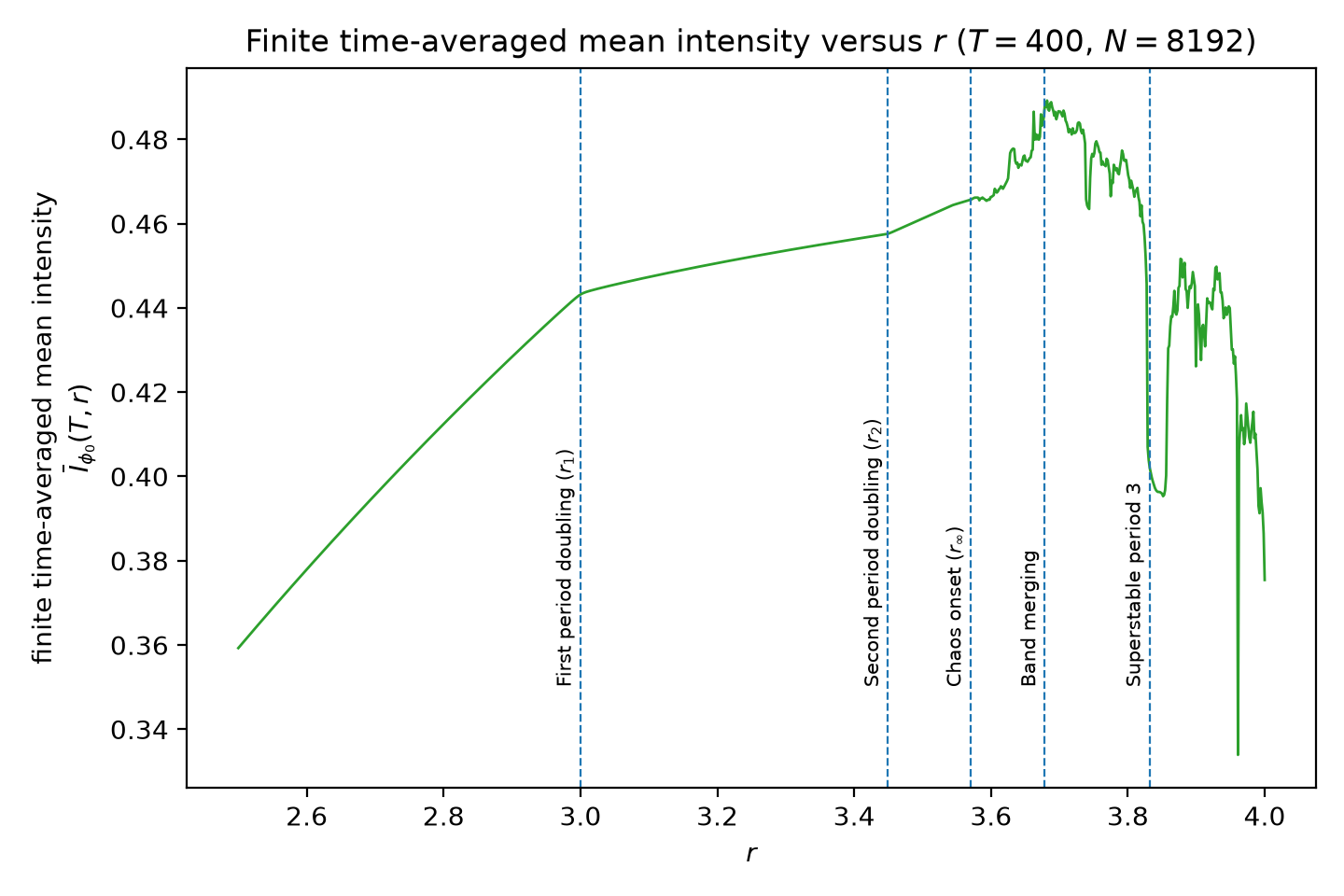}
 \caption{Finite time-averaged mean intensity
 $\overline{I}_{\phi_0}(T,r)$ for $T=400$ and $N=8192$.  The five vertical
 lines mark the first period doubling $r_1=3$, the second period doubling
 $r_2=1+\sqrt6$, chaos onset $r_\infty\simeq3.569945672$, the first
 band-merging point $r_{\mathrm{bm},1}\simeq3.6785735104$, and the
 superstable period-three parameter
 $r_{\mathrm{ss},3}\simeq3.8318740553$.  Fine structure in the chaotic
 regime is a finite-resolution numerical result; no infinite-time or
 continuum limit is claimed.}
 \label{fig:mean-intensity-vs-r}
\end{figure}

Because $X_n$ is multiplication by $p_n$, one may also define the finite
instantaneous moments
\begin{equation}
 \begin{aligned}
 M_{2,\Psi}(n,r)&=\langle\Psi,X_n^2\Psi\rangle
 =\int_0^1|\Psi(u)|^2p_n(u;r)^2\,du,\\
 M_{4,\Psi}(n,r)&=\langle\Psi,X_n^4\Psi\rangle
 =\int_0^1|\Psi(u)|^2p_n(u;r)^4\,du.
 \end{aligned}
 \label{eq:second-fourth-intensity-moments}
\end{equation}
and the instantaneous normalized second-order intensity moment
\begin{equation}
 G_\Psi^{(2)}(n,r)=\frac{M_{4,\Psi}(n,r)}{M_{2,\Psi}(n,r)^2}.
 \label{eq:instantaneous-normalized-second-order}
\end{equation}
Writing $\overline M_{j,\Psi}(T,r)=(T+1)^{-1}
\sum_{n=0}^{T}M_{j,\Psi}(n,r)$, the pooled finite-time quantity plotted in
Figure~\ref{fig:normalized-second-order-moment} is
\begin{equation}
 \overline G_\Psi^{(2)}(T,r)
 =\frac{\overline M_{4,\Psi}(T,r)}{\overline M_{2,\Psi}(T,r)^2}
 =1+\frac{\operatorname{Var}(p_n^2)}{\overline M_{2,\Psi}(T,r)^2}
 \geq1,
 \label{eq:pooled-normalized-second-order}
\end{equation}
where the variance is taken with respect to the normalized state density and
the uniform distribution on $n=0,\ldots,T$.  Thus this is a ratio of pooled
moments, not the arithmetic mean of the instantaneous ratios.

This diagnostic is only formally analogous to an optical $g^{(2)}$.  No
annihilation or creation operator, normally ordered photon correlator,
physical number operator, or photon-counting model has been introduced.
Consequently, terms such as ``sub-Poissonian statistics'' and
``antibunching'' would be only formal analogies and cannot occur here as
values below one.

\begin{figure}[htbp]
 \centering
 \includegraphics[width=0.96\textwidth]{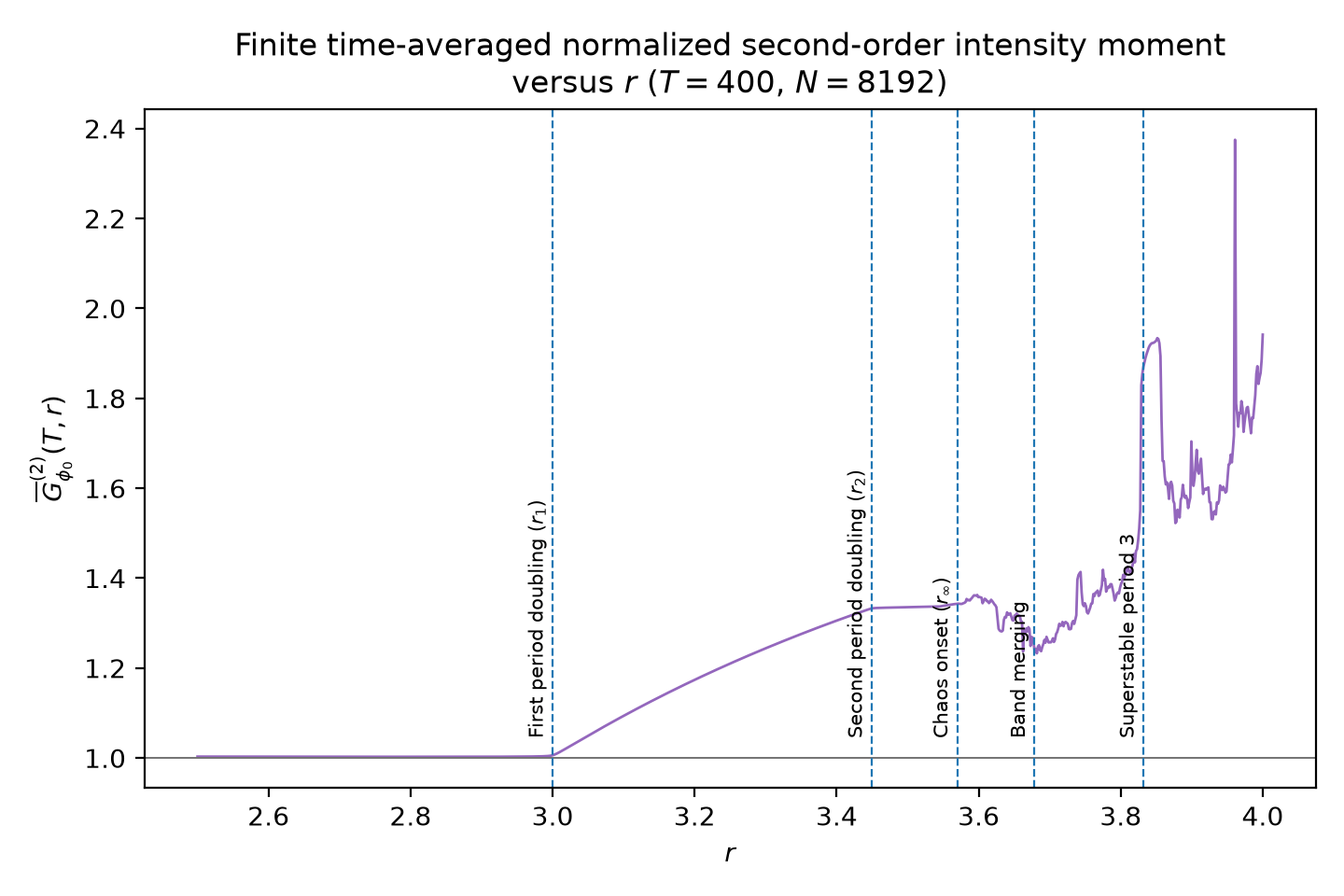}
 \caption{Finite time-averaged normalized second-order intensity moment
 $\overline G_{\phi_0}^{(2)}(T,r)$ for $\Psi=\phi_0$, $T=400$, deterministic
 equal-stratum quadrature with $N=8192$ and midpoint offset $s=1/2$, and
 1001 equally spaced values on $r\in[2.5,4]$.  The vertical lines mark the
 first period doubling $r_1=3$, second period doubling $r_2=1+\sqrt6$,
 chaos onset $r_\infty\simeq3.569945672$, first band-merging point
 $r_{\mathrm{bm},1}\simeq3.6785735104$, and superstable period-three
 parameter $r_{\mathrm{ss},3}\simeq3.8318740553$.  The curve is a finite
 deterministic numerical result and its chaotic fine structure is
 finite-resolution; no infinite-time, continuum, or photon-statistical
 conclusion is claimed.}
 \label{fig:normalized-second-order-moment}
\end{figure}

Finally, a normalized two-time commutator correlation matrix is considered.
Out-of-time-order correlators originated in a quasiclassical setting and
have since become widely used diagnostics of quantum-chaotic growth and
scrambling~\cite{LarkinOvchinnikov,MSS,Rozenbaum2017,ShenkerStanford,Swingle}.
Let $\Pi=-i\,d/du$ denote the formal conjugate differential operator.  Since
its realization on $[0,1]$ requires a domain choice, the following identity
is used only on a suitable common smooth test domain in the open interval:
\begin{equation}
 [X_n,\Pi]=iM_{p_n'}.
 \label{eq:scalar-commutator-identity}
\end{equation}
For a normalized state $\Psi$, define
\begin{align}
 C_{n\ell}(r;\Psi)
 &=\langle\Psi,[X_n,\Pi]^\dagger[X_\ell,\Pi]\Psi\rangle
 =\int_0^1|\Psi(u)|^2p_n'(u;r)p_\ell'(u;r)\,du,
 \label{eq:scalar-commutator-correlation}\\
 \rho_{n\ell}(r;\Psi)
 &=\frac{C_{n\ell}(r;\Psi)}
 {\sqrt{C_{nn}(r;\Psi)C_{\ell\ell}(r;\Psi)}}.
 \label{eq:scalar-normalized-commutator-correlation}
\end{align}
Thus $C$ is a two-time commutator correlation matrix and $\rho$ is the
normalized Gram matrix of the weighted amplitudes $p_n'$.  When the diagonal
normalizations are positive, $\rho$ is symmetric with unit diagonal,
$|\rho_{n\ell}|\leq1$, and is positive semidefinite.  Independent positive
rescaling of each derivative vector leaves $\rho$ unchanged; this permits the
stable tangent-shape normalization used for
Figure~\ref{fig:scalar-otoc-normalized} without reconstructing exponentially
large discarded amplitudes.

This matrix may also be called a normalized scalar OTOC-type matrix, but it is
not asserted to be the unique or standard four-point OTOC convention.  The
calculation is finite in $n,\ell$ and in quadrature resolution.  Its displayed
geometry is a finite numerical observation, not evidence for an asymptotic
OTOC limit, quantum chaos, many-body scrambling, or operator growth in a
many-body sense.

\begin{figure}[htbp]
 \centering
 \includegraphics[width=\textwidth]{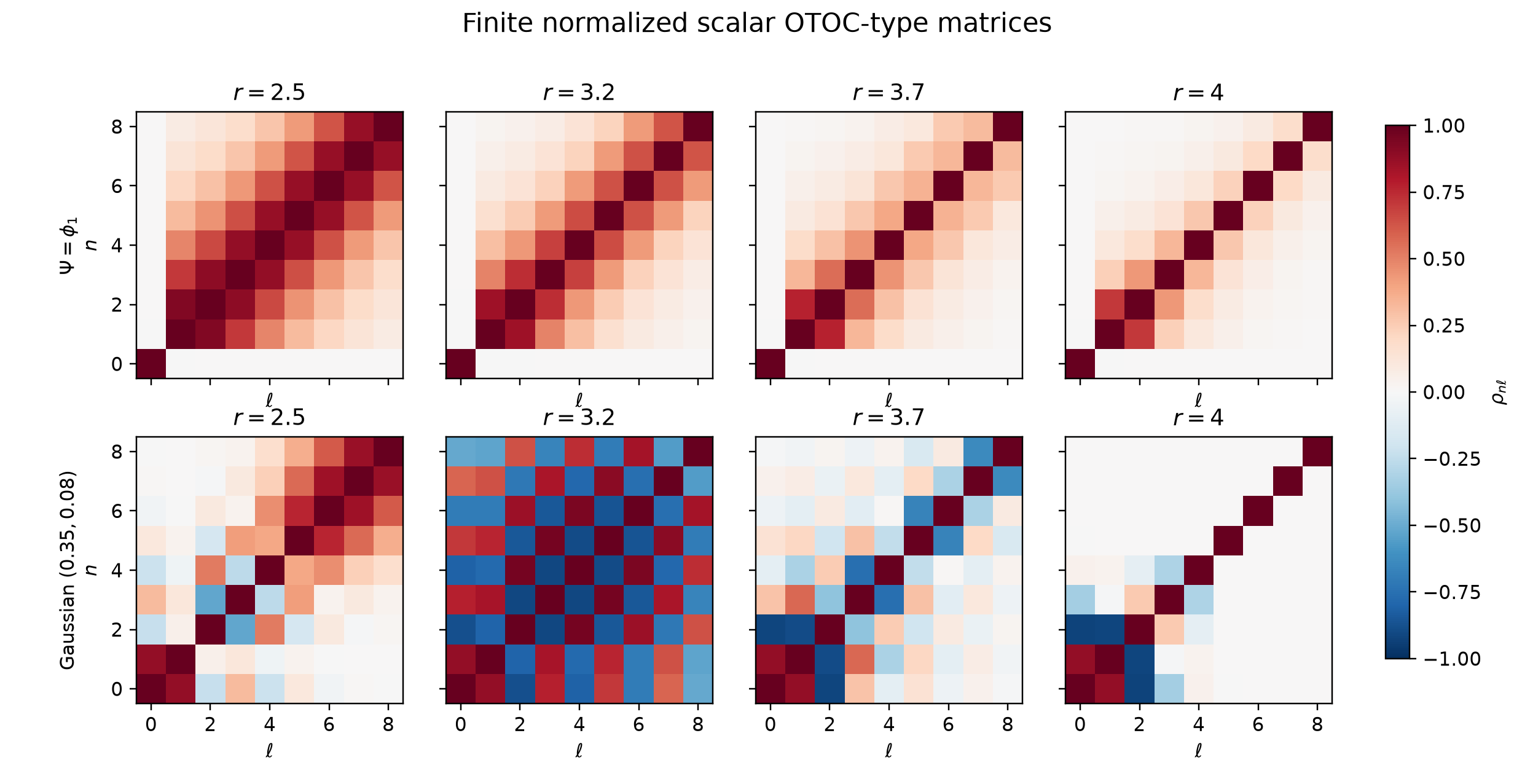}
 \caption{Finite normalized scalar OTOC-type matrices
 $\rho_{n\ell}(r;\Psi)$ for $\Psi=\phi_1$ (top row) and the normalized
 Gaussian density with center $0.35$ and width $0.08$ (bottom row), at
 $r=5/2,16/5,37/10$, and $4$.  The displayed indices are
 $n,\ell=0,\ldots,8$.  Each panel uses deterministic equal-stratum
 quadrature with $N=2^{18}$ and midpoint offset $s=1/2$; offsets
 $1/4,1/2,3/4$ were included in the resolution audit.  All panels use the
 shared scale $[-1,1]$ and were computed as normalized derivative-vector
 Gram matrices.  The patterns are finite-resolution observations; no
 continuum or asymptotic limit, many-body scrambling, quantum-chaos, or
 state-independence conclusion is claimed.}
 \label{fig:scalar-otoc-normalized}
\end{figure}

\FloatBarrier

\section{Finite-dimensional operator-valued extension}
\label{sec:finite-operator-extension}

The separate finite-matrix recursion
\begin{equation}
 X_{k+1}=R X_k(I-X_k)R^\dagger,
 \qquad X_0=P_D M_uP_D,\qquad D=32.
 \label{eq:operator-logistic}
\end{equation}
is now considered.  It is not the multiplication-operator sequence
$X_n=p_n(M_u)$ in Eq.~\eqref{eq:multiplication-operator-iterates}.
Every statement in this section concerns only the displayed dimension and
finite horizon $k=0,\ldots,40$.  Neither an infinite-dimensional limit nor a general boundedness or
asymptotic-stability theorem is established.  No clipping,
eigenvalue projection, damping, or other numerical stabilization is applied.
A trajectory is stopped before storing a non-finite iterate or one with
$\|X_k\|_2>10^{12}$.

For basis indices $n=0,\ldots,D-1$, the three real amplitude profiles are
\begin{align}
 d_A(n)&=2e^{-\gamma_A n}, & \gamma_A&=0.10,\nonumber\\
 d_B(n)&=2e^{-\gamma_B n}\cos(\omega n),
   & \gamma_B&=0.08,\quad \omega=\pi/4,\label{eq:operator-r-profiles}\\
 d_C(n)&=2e\gamma_C n e^{-\gamma_C n}, & \gamma_C&=0.12.\nonumber
\end{align}
Here $d_n$ is an amplitude of $R$, while $|d_n|^2$ is its local quadratic
scale; neither is, in general, a scalar logistic parameter.

\begin{figure}[htbp]
 \centering
 \includegraphics[width=0.98\textwidth]{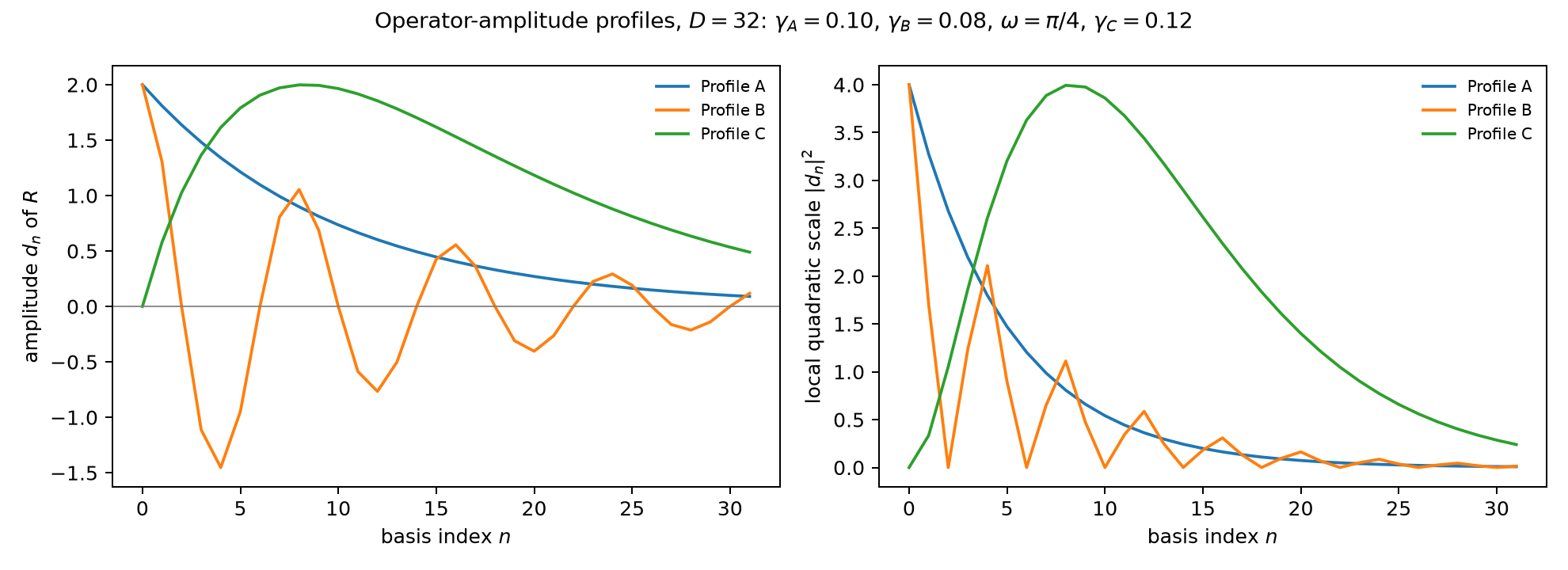}
 \caption{The three operator-amplitude profiles in
 Eq.~\eqref{eq:operator-r-profiles} for $D=32$ (left) and their squared
 magnitudes (right).  Parameters are $\gamma_A=0.10$, $\gamma_B=0.08$,
 $\omega=\pi/4$, and $\gamma_C=0.12$.  The quantities $d_n$ are amplitudes
 of $R$; $|d_n|^2$ are local quadratic scales.}
 \label{fig:operator-r-profiles}
\end{figure}

\subsection{Diagonal and tridiagonal reconnaissance}

For the diagonal construction $R=\operatorname{diag}(d_0,\ldots,d_{D-1})$,
all three trajectories complete through $k=40$.  With the normalized maximum
trajectory distance defined in the numerical audit, the pairwise values are
$D_{AB}=0.37114$, $D_{AC}=0.82778$, and $D_{BC}=0.83281$; hence all three
pairs are classified as clearly distinct under the predeclared threshold
$D_{PQ}\geq0.25$.  The computed spectra remain within approximately
$[0,1]$, apart from negative residuals no larger than $6.7\times10^{-17}$
in magnitude.  The finite spectral, bandwidth, and successive-commutator
diagnostics visibly differ, but these observations are not asymptotic results.

The same trajectories were evaluated in the shifted-Legendre probe
$\Psi_F=\phi_1$ and in the normalized $D$-mode projection of the Gaussian
wave packet with center $0.35$ and width $0.08$.  Their finite expectation
ranges are
\begin{center}
\begin{tabular}{c|cc}
 profile & $\langle X_k\rangle_{\phi_1}$ &
 $\langle X_k\rangle_{\rm Gaussian}$\\ \hline
 A & $[0.327492,0.818330]$ & $[0.320681,0.729240]$\\
 B & $[0.170429,0.500000]$ & $[0.088750,0.487102]$\\
 C & $[0,0.500000]$ & $[0.164438,0.350002]$
\end{tabular}
\end{center}

For the mode-mixing construction, the entries are chosen as
\begin{equation}
 R_{n,n+1}=R_{n+1,n}
 =\epsilon\sqrt{|d_n d_{n+1}|}\,\sgn(d_n+d_{n+1}),
 \qquad \sgn(0)=0,
 \label{eq:operator-r-tridiagonal}
\end{equation}
in addition to the diagonal entries $R_{nn}=d_n$.  The dimensionless
$\epsilon$ is a nearest-neighbour coupling-strength parameter introduced for
this extension: $\epsilon=0$ gives the diagonal construction, and it is not a
scalar logistic parameter.  The principal reconnaissance uses
$\epsilon=0.10$.

At this coupling, profile C crosses the norm threshold at $k=9$ (last stored
$k=8$), profile A crosses at $k=13$ (last stored $k=12$), and profile B
crosses at $k=16$ (last stored $k=15$).  Before termination the iterates
develop large negative eigenvalues, rapidly increasing spectral norms, and
rapidly increasing successive commutator norms.  This is finite-horizon
instability for these particular matrices and parameters; it does not prove
instability for all tridiagonal $R$.

Because $X_0$ is already tridiagonal, its total off-diagonal fraction cannot
measure newly generated mode spreading.  Instead, the following quantity is audited:
\begin{equation}
 \eta_k^{(>1)}=
 \frac{\left(\sum_{|i-j|>1}|(X_k)_{ij}|^2\right)^{1/2}}
      {\|X_k\|_F}.
 \label{eq:beyond-initial-band}
\end{equation}
The combined $\gamma$--$\epsilon$ scan contains finite parameter cells that
complete through $k=40$.  Among those cells,
$\max_{1\leq k\leq40}\eta_k^{(>1)}$ ranges from $0.11082$ to $0.57302$,
so their finite trajectories retain nonzero weight outside the initial band.
The scan remains a parameter-selection diagnostic and does not establish
asymptotic stability.
Entropy-based summaries could provide a separate mode-spreading diagnostic,
but their definition and interpretation would require care, including the
distinction between R{\'e}nyi and Shannon entropies~\cite{ZyczkowskiEntropy}.

\begin{figure}[htbp]
 \centering
 \includegraphics[width=0.98\textwidth]{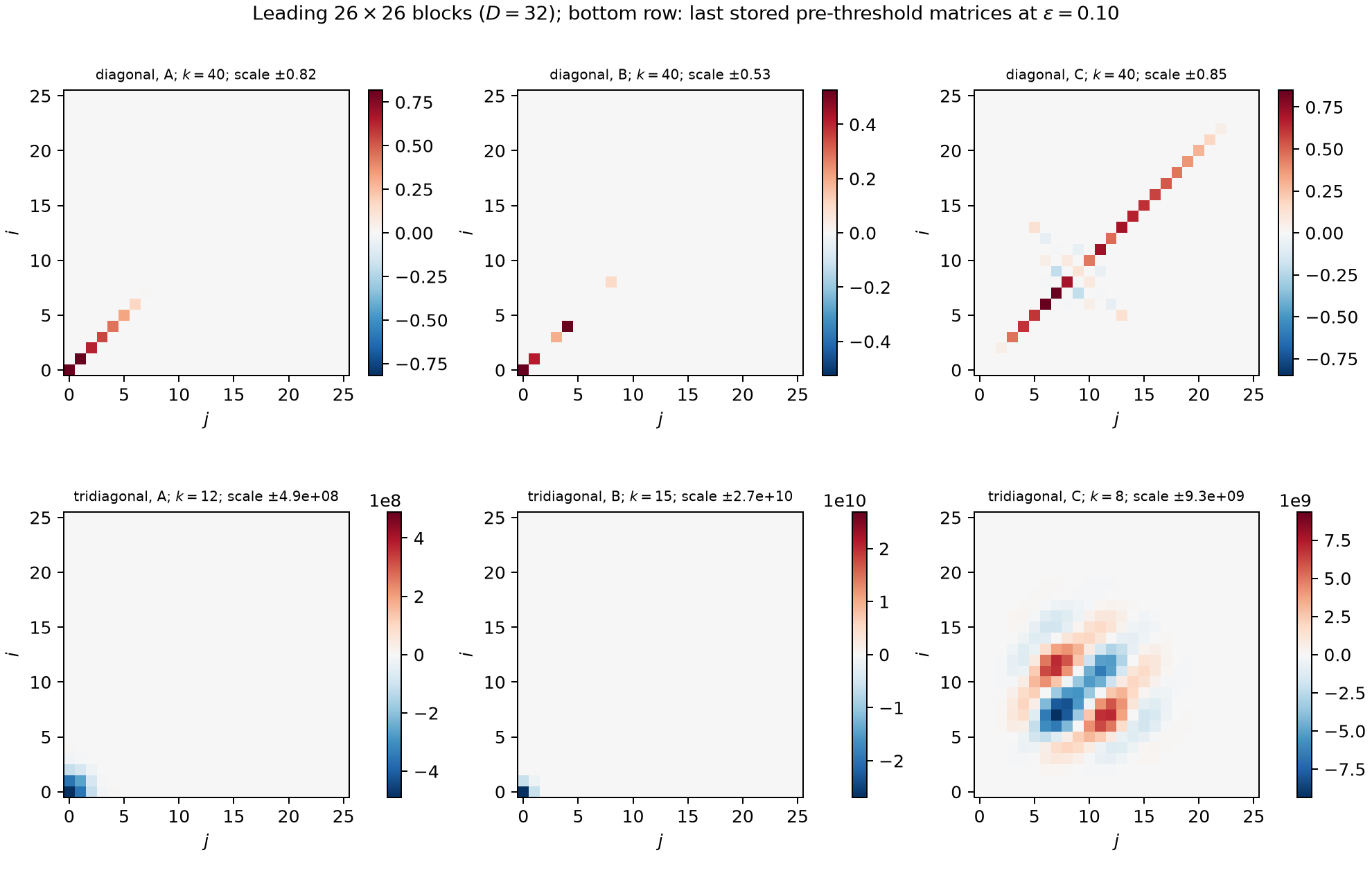}
 \caption{Leading $26\times26$ blocks of the $D=32$ Hermitian iterates.
 The top row shows diagonal profiles A, B, and C at $k=40$.  The bottom row
 shows tridiagonal profiles A, B, and C for $\epsilon=0.10$ at their last
 stored pre-threshold steps $k=12,15,8$, respectively.  Each panel has its
 own symmetric color normalization and states its numerical scale; matrix
 entries span very different amplitudes, so colors are not directly
 comparable between panels.  The bottom row visualizes finite pre-threshold
 structure, not bounded long-time dynamics.}
 \label{fig:operator-final-heatmaps}
\end{figure}

\FloatBarrier

\subsection{Outlook}

Natural next steps are dimension refinement, controlled
$\gamma$--$\epsilon$ parameter studies, sufficient conditions for preserving
$0\leq X_k\leq I$, and rigorous boundedness criteria.  Stronger definitions
of operator OTOCs require a specified operator domain and finite-dimensional
approximation, while any infinite-dimensional version of
Eq.~\eqref{eq:operator-logistic} would require a separate convergence and
domain analysis.

\FloatBarrier

\section{Concluding remarks}

The main construction associates the scalar logistic iterates $p_n$ with
the multiplication operators $X_n=p_n(M_u)$ on $L^2([0,1])$, where $M_u$
is multiplication by the coordinate.  The fixed matrix elements of these
operators in the normalized shifted-Legendre basis have been studied.  This
formulation preserves the scalar functional dynamics while making the mode
of every convergence statement explicit.

Three regimes admit analytical conclusions.  For $r=5/2$, scalar convergence
to the attracting fixed point holds almost everywhere and dominated
convergence gives $x_{kl}(n;5/2)\to3\delta_{kl}/5$ for every fixed pair
$k,l$.  For $r=16/5$, the exact phase-basin decomposition of the attracting
two-cycle gives distinct even and odd fixed-element limits and their sum
$21\delta_{kl}/16$.  For $r=4$, the conjugacy to angle doubling leads to an
exact finite Chebyshev-moment representation and the fixed-element estimate
$x_{kl}(n;4)=\delta_{kl}/2+O_{kl}(4^{-n})$.  None of these results is an
operator-norm convergence statement, nor are the estimates uniform over
basis indices that grow with time.

The remaining scalar results are finite-resolution, finite-time diagnostics.
At $r=37/10$, deterministic grid, offset, and horizon refinements show that
finite Ces\`aro and specified late-window averages are better resolved than raw
late matrix elements, without establishing an invariant-measure or
infinite-time result.  The numerical presentation also includes selected
matrix-element time series, a bifurcation-style $x_{22}$ diagram with one
classical finite-time Lyapunov estimate, a finite time-averaged mean
intensity, a pooled normalized second-order intensity moment, and normalized
scalar OTOC-type Gram matrices.  Their stated quadrature resolutions and
horizons are part of their definitions as reported here.

The separate finite-dimensional recursion
$X_{k+1}=R X_k(I-X_k)R^\dagger$ was also explored at $D=32$.  Three pairwise distinct diagonal
profiles completed through $k=40$.  With the selected tridiagonal profiles
and coupling $\epsilon=0.10$, the numerical norm threshold was crossed at
finite steps.  A combined $\gamma$--$\epsilon$ scan nevertheless contains
finite parameter cells completing through $k=40$ while retaining weight
outside the initial tridiagonal band.  These observations prove neither
asymptotic boundedness nor instability, and they provide no
infinite-dimensional limit.

Natural next steps are dimension refinement, rigorous quadrature and
continuum-control analysis, sufficient conditions for preserving
$0\leq X_k\leq I$, and boundedness criteria for the matrix recursion.
Operator OTOCs require a specified operator domain, and any
infinite-dimensional formulation requires mathematically controlled
approximations and convergence arguments.  Such work should continue to
separate finite-horizon completion from asymptotic stability, fixed-dimension
evidence from continuum conclusions, and numerical observations from proved
statements.

\section*{Acknowledgment}

During the preparation of this work, generative AI tools were used
in two distinct capacities. First, AI-assisted code generation
(ChatGPT Work and Codex) was used for creation of Python routines;
all final implementation was written, inspected, and verified by the
authors. Second, a large language model (ChatGPT Work) was used to assist
with prose editing and language refinement at the manuscript stage.
The authors reviewed and took responsibility for all content thus
produced. No AI tool is listed as an author; all authors accept full
responsibility for the scientific content of this work, consistent
with both arXiv and Reports on Mathematical Physics policies.

\bibliography{references_checked}

\end{document}